\documentclass[reqno,11pt]{amsart}
\usepackage{amsmath,amsthm,amssymb,amsfonts}
\usepackage{epsfig,color,graphicx,enumerate}
\usepackage[notcite,notref]{showkeys}
\usepackage{accents}
\newcommand*{\dt}[1]{%
\accentset{\mbox{\large\bfseries .}}{#1}}

\numberwithin{equation}{section}
\theoremstyle{plain}
\newtheorem{teo}{Theorem}[section]
\newtheorem{lemma}[teo]{Lemma}
\newtheorem{cor}[teo]{Corollary}

\newtheorem{prop}[teo]{Proposition}
\newtheorem{deff}[teo]{Definition}
\theoremstyle{remark}
\newtheorem{oss}[teo]{Remark}
\newtheorem{exam}[teo]{Example}

\def\E{\mathcal E}

\def\M{\mathcal M}

\def\V{\mathcal V}

\def\N{\mathbb N}
\def\R{\mathbb R}

\def\O{\Omega}
\def\eps{\varepsilon}

\def\bal{\begin{aligned}}
\def\eal{\end{aligned}}
\def\ds{\displaystyle}

\newcommand{\be}{\begin{equation}}
\newcommand{\ee}{\end{equation}}
\newcommand{\bib}[4]{\bibitem{#1}{\sc#2: }{\it#3. }{#4.}}

\newcommand{\cp}{\mathop{\rm cap}\nolimits}

\title[Optimal Potentials for Schr\"odinger Operators]{Optimal Potentials for Schr\"odinger Operators}

\author{G. Buttazzo}
\address{Dipartimento di Matematica, Universit\`a di Pisa, Largo B. Pontecorvo 5, 56126 Pisa, ITALY}
\email{buttazzo@dm.unipi.it}

\author{A. Gerolin}
\address{Dipartimento di Matematica, Universit\`a di Pisa, Largo B. Pontecorvo 5, 56126 Pisa, ITALY}
\email{gerolin@mail.dm.unipi.it }

\author{B. Ruffini}
\address{Scuola Normale Superiore di Pisa, Piazza dei Cavalieri 7, 56126 Pisa, ITALY}
\email{berardo.ruffini@sns.it}

\author{B. Velichkov}
\address{Scuola Normale Superiore di Pisa, Piazza dei Cavalieri 7, 56126 Pisa, ITALY}
\email{b.velichkov@sns.it}

\begin{document}

\maketitle

\begin{abstract}
We consider the Schr\"odinger operator $-\Delta+V(x)$ on $H^1_0(\O)$, where $\O$ is a given domain of $\R^d$. Our goal is to study some optimization problems where an optimal potential $V\ge0$ has to be determined in some suitable admissible classes and for some suitable optimization criteria, like the energy or the Dirichlet eigenvalues.
\end{abstract}
\medskip
\textbf{Keywords:} Schr\"odinger operators, optimal potentials, spectral optimization, capacity.

\textbf{2010 Mathematics Subject Classification:} 49J45, 35J10, 49R05, 35P15, 35J05.

\section{Introduction}\label{sintro}

In this paper we consider the Schr\"odinger operator $-\Delta+V(x)$ on $H^1_0(\O)$, where $\O$ is a given domain of $\R^d$. Our goal is to study some optimization problems where an optimal potential $V\ge0$ has to be determined, for some suitable optimization criteria, among the ones belonging to some admissible classes. The problems we are dealing with are then
$$\min\big\{F(V)\ :\ V\in\V\big\},$$
where $F$ denotes the cost functional and $\V$ the admissible class. The cost functionals we aim to include in our framework are for instance the following.

\medskip{\it Integral functionals.} Given a right-hand side $f\in L^2(\O)$ we consider the solution $u_V$ of the elliptic PDE
$$-\Delta u+Vu=f\hbox{ in }\O,\qquad u\in H^1_0(\O).$$
The integral cost functionals we may consider are of the form
$$F(V)=\int_\O j\big(x,u_V(x),\nabla u_V(x)\big)\,dx,$$
where $j$ is a suitable integrand that we assume convex in the gradient variable and bounded from below. One may take, for example, 
$$j(x,s,z)\ge-a(x)-c|s|^2,$$
with $a\in L^1(\O)$ and $c$ smaller than the first Dirichlet eigenvalue of the Laplace operator $-\Delta$ in $\O$. In particular, the energy $\E_f(V)$ defined by
\be\label{energy}
\E_f(V)=\inf\left\{\int_\O\Big(\frac12|\nabla u|^2+\frac12V(x)u^2-f(x)u\Big)\,dx\ :\ u\in H^1_0(\O)\right\},
\ee
belongs to this class since, integrating by parts its Euler-Lagrange equation, we have
$$\E_f(V)=-\frac12\int_\O f(x)u_V\,dx,$$
which corresponds to the integral functional above with
$$j(x,s,z)=-\frac12 f(x)s.$$

\medskip{\it Spectral functionals.} For every admissible potential $V\ge0$ we consider the spectrum $\Lambda(V)$ of the Schr\"odinger operator $-\Delta+V(x)$ on $H^1_0(\O)$. If $\O$ is bounded or has finite measure, or if the potential $V$ satisfies some suitable integral properties, the operator $-\Delta+V(x)$ has a compact resolvent and so its spectrum $\Lambda(V)$ is discrete:
$$\Lambda(V)=\big(\lambda_1(V),\lambda_2(V),\dots\big),$$
where $\lambda_k(V)$ are the eigenvalues counted with their multiplicity. The spectral cost functionals we may consider are of the form
$$F(V)=\Phi\big(\Lambda(V)\big),$$
for a suitable function $\Phi:\R^\N\to\overline{\R}$. For instance, taking $\Phi(\Lambda)=\lambda_k$ we obtain
$$F(V)=\lambda_k(V).$$

\medskip

Concerning the admissible classes we deal with, we consider mainly the cases
$$\V=\left\{V\ge0\ :\ \int_\O V^p\,dx\le1\right\}
\qquad\hbox{and}\qquad
\V=\left\{V\ge0\ :\ \int_\O V^{-p}\,dx\le1\right\};$$
in some situations more general admissible classes $\V$ will be considered, see Theorem \ref{main} and Theorem \ref{mainPhi}.

In Section \ref{s31} our assumptions allow to take $F(V)=-\E_f(V)$ and thus the optimization problem becomes the maximization of $\E_f$ under the constraint $\int_\O V^p\,dx\le1$. We prove that for $p\ge1$, there exists an optimal potential for the problem
\be
\max\left\{\E_f(V)\ :\ \int_\O V^p\,dx\le1\right\}.
\ee 
The existence result is sharp in the sense that for $p<1$ the maximum cannot be achieved (see Remark \ref{controesempio}). For the existence issue in the case of a bounded domain, we follow the ideas of Egnell \cite{egnell}, summarized in \cite[Chapter 8]{hen06} (where a complete reference for the problem can also be found). The case $p=1$ is particularly interesting and we show that in this case the optimal potentials are of the form
$$V_{opt}=\frac{f}{M}\left(\chi_{\omega_+}-\chi_{\omega_-}\right),$$
where $\chi_U$ indicates the characteristic function of the set $U$, $f\in L^2(\O)$, $M=\|u_V\|_{L^\infty(\O)}$, and $\omega_\pm=\{u=\pm M\}$. In Section \ref{s4} we deal with minimization problems of the form
\be\label{1.3}
\min\big\{F(V)\ :\ \int_\O V^{-p}\,dx\le1\big\}.
\ee
We prove a general result (Theorem \ref{mainPhi}) establishing the existence of an optimal potential under some mild conditions on the functional $F$. In particular, we obtain the existence of optimal potentials for a large class of spectral and energy functionals (see Corollary \ref{lb}).

In Section \ref{s5} we deal with the case of unbounded domains $\O$. precisely, we prove that in the case $\O=\R^d$ and $F=\E_f$ or $F=\lambda_1$, the solutions of problem \eqref{1.3} exist and are such that $1/V$ is compactly supported, provided $f$ is compactly supported. Finally, in Section \ref{s6} we make some further remarks and present some open questions.

\section{Capacitary measures and $\gamma$-convergence}\label{s2}

For a subset $E\subset\R^d$ its {\it capacity} is defined by
$$\cp(E)=\inf\left\{\int_{\R^d}|\nabla u|^2\,dx+\int_{\R^d}u^2\,dx\ :\ u\in H^1(\R^d),\ u\ge 1\ \hbox{in a neighborhood of } E\right\}.$$
If a property $P(x)$ holds for all $x\in\O$, except for the elements of a set $E\subset\O$ of capacity zero, we say that $P(x)$ holds {\it quasi-everywhere} (shortly {\it q.e.}) in $\O$, whereas the expression {\it almost everywhere} (shortly {\it a.e.}) refers, as usual, to the Lebesgue measure, which we often denote by $|\cdot|$. 

A subset $A$ of $\R^d$ is said to be {\it quasi-open} if for every $\eps>0$ there exists an open subset $A_\eps$ of $\R^d$, with $A\subset A_\eps$, such that $\cp(A_\eps\setminus A)<\eps$. Similarly, a function $u:\R^d\to\R$ is said to be {\it quasi-continuous} (respectively {\it quasi-lower semicontinuous}) if there exists a decreasing sequence of open sets $(A_n)_n$ such that $\cp(A_n)\to0$ and the restriction $u_n$ of $u$ to the set $A_n^c$ is continuous (respectively lower semicontinuous). It is well known (see for instance \cite{evgar}) that every function $u\in H^1(\R^d)$ has a quasi-continuous representative $\widetilde u$, which is uniquely defined up to a set of capacity zero, and given by
$$\widetilde u(x)=\lim_{\eps\to0}\frac{1}{|B_\eps(x)|}\int_{B_\eps(x)}u(y)\,dy\,,$$
where $B_\eps(x)$ denotes the ball of radius $\eps$ centered at $x$. We identify the (a.e.) equivalence class $u\in H^1(\R^d)$ with the (q.e.) equivalence class of quasi-continuous representatives $\widetilde u$.

We denote by $\M^+(\R^d)$ the set of positive Borel measures on $\R^d$ (not necessarily finite or Radon) and by $\M^+_{\cp}(\R^d)\subset\M^+(\R^d)$ the set of {\it capacitary measures}, i.e. the measures $\mu\in\M^+(\R^d)$ such that $\mu(E)=0$ for any set $E\subset\R^d$ of capacity zero. We note that when $\mu$ is a capacitary measure, the integral $\int_{\R^d} |u|^2\,d\mu$ is well-defined for each $u\in H^1(\R^d)$, i.e. if $\widetilde u_1$ and $\widetilde u_2$ are two quasi-continuous representatives of $u$, then $\int_{\R^d} |\widetilde u_1|^2\,d\mu=\int_{\R^d} |\widetilde u_2|^2\,d\mu$. 

For a subset $\O\subset\R^d$, we define the Sobolev space $H^1_0(\O)$ as
\be
H^1_0(\O)=\left\{u\in H^1(\R^d):\ u=0\ \hbox{q.e. on }\O^c\right\}.
\ee
Alternatively, by using the capacitary measure $I_\O$ defined as
\be\label{Iomega}
I_\O(E)=\begin{cases}
0&\hbox{if }\cp(E\setminus\O)=0\\
+\infty&\hbox{if }\cp(E\setminus\O)>0
\end{cases}
\qquad\hbox{for every Borel set }E\subset\R^d,
\ee
the Sobolev space $H^1_0(\O)$ can be defined as
$$H^1_0(\O)=\left\{u\in H^1(\R^d)\ :\ \int_{\R^d}|u|^2\,dI_\O<+\infty\right\}.$$
More generally, for any capacitary measure $\mu\in\M^+_{\cp}(\R^d)$, we define the space
$$H^1_\mu=\left\{u\in H^1(\R^d)\ :\ \int_{\R^d}|u|^2\,d\mu<+\infty\right\},$$
which is a Hilbert space when endowed with the norm $\|u\|_{1,\mu}$, where
$$\|u\|_{1,\mu}^2=\int_{\R^d}|\nabla u|^2\,dx+\int_{\R^d}u^2\,dx+\int_{\R^d} u^2\,d\mu.$$
If $u\notin H^1_\mu$, then we set $\|u\|_{1,\mu}=+\infty$.

For $\O\subset\R^d$, we define $\M^+_{\cp}(\O)$ as the space of capacitary measures $\mu\in\M^+_{\cp}(\R^d)$ such that $\mu(E)=+\infty$ for any set $E\subset\R^d$ such that $\cp(E\setminus\O)>0$. For $\mu\in\M^+_{\cp}(\R^d)$, we denote with $H^1_\mu(\O)$ the space $H^1_{\mu\vee I_\O}=H^1_\mu\cap H^1_0(\O)$. 

\begin{deff}\label{Gamma}
Given a metric space $(X,d)$ and sequence of functionals $J_n:X\to\R\cup\{+\infty\}$, we say that $J_n$ $\Gamma$-converges to the functional $J:X\to\R\cup\{+\infty\}$, if the following two conditions are satisfied:
\begin{enumerate}[(a)]
\item for every sequence $x_n$ converging in to $x\in X$, we have
$$J(x)\le\liminf_{n\to\infty}J_n(x_n);$$
\item for every $x\in X$, there exists a sequence $x_n$ converging to $x$, such that
$$J(x)=\lim_{n\to\infty}J_n(x_n).$$
\end{enumerate}
\end{deff}

For all details and properties of $\Gamma$-convergence we refer to \cite{dm93}; here we simply recall that, whenever $J_n$ $\Gamma$-converges to $J$,
\be
\min_{x\in X}J(x)\le\liminf_{n\to\infty}\min_{x\in X}J_n(x).
\ee

\begin{deff}\label{gamma}
We say that the sequence of capacitary measures $\mu_n\in\M^+_{\cp}(\O)$, $\gamma$-converges to the capacitary measure $\mu\in\M^+_{\cp}(\O)$ if the sequence of functionals $\|\cdot\|_{1,\mu_n}$ $\Gamma$-converges to the functional $\|\cdot\|_{1,\mu}$ in $L^2(\O)$, i.e. if the following two conditions are satisfied:
\begin{itemize}
\item for every sequence $u_n\to u$ in $L^2(\O)$ we have
$$\int_{\R^d}|\nabla u|^2\,dx+\int_{\R^d} u^2\,d\mu\le\liminf_{n\to\infty}\left\{\int_{\R^d}|\nabla u_n|^2\,dx+\int_{\R^d} u_n^2\,d\mu_n\right\};$$
\item for every $u\in L^2(\O)$, there exists $u_n\to u$ in $L^2(\O)$ such that
$$\int_{\R^d}|\nabla u|^2\,dx+\int_{\R^d} u^2\,d\mu=\lim_{n\to\infty}\left\{\int_{\R^d}|\nabla u_n|^2\,dx+\int_{\R^d} u_n^2\,d\mu_n\right\}.$$
\end{itemize}
\end{deff}

If $\mu\in\M^+_{\cp}(\O)$ and $f\in L^2(\O)$ we define the functional $J_\mu(f,\cdot):L^2(\O)\to\R\cup\{+\infty\}$ by
\be\label{F}
J_\mu(f,u)=\frac12\int_\O|\nabla u|^2\,dx+\frac12\int_\O u^2\,d\mu-\int_\O fu\,dx.
\ee

If $\O\subset\R^d$ is a bounded open set, $\mu\in\M^+_{\cp}(\O)$ and $f\in L^{2}(\O)$, then the functional $J_\mu(f,\cdot)$ has a unique minimizer $u\in H^1_\mu$ that verifies the PDE formally written as
\be\label{dumuf}
-\Delta u+\mu u=f,\qquad u\in H^1_\mu(\O),
\ee
and whose precise meaning is given in the weak form
$$\begin{cases}
\ds\int_\O\nabla u\cdot\nabla\varphi\,dx+\int_\O u\varphi\,d\mu
=\int_{\O}f\varphi\,dx,\qquad\forall\varphi\in H^1_\mu(\O),\\
u\in H^1_\mu(\O). 
\end{cases}$$
The resolvent operator of $-\Delta+\mu$, that is the map $\mathcal{R}_\mu$ that associates to every $f\in L^2(\O)$ the solution $u\in H^1_\mu(\O)\subset L^2(\O)$, is a compact linear operator in $L^2(\O)$ and so, it has a discrete spectrum 
$$0< \dots \le \Lambda_k\le\dots\le\Lambda_2\le\Lambda_1.$$
Their inverses $1/\Lambda_k$ are denoted by $\lambda_k(\mu)$ and are the eigenvalues of the operator $-\Delta+\mu$.

In the case $f=1$ the solution will be denoted by $w_\mu$ and when $\mu=I_\O$ we will use the notation $w_\O$ instead of $w_{I_\O}$. We also recall (see \cite{bubu05}) that if $\O$ is bounded, then the strong $L^2$-convergence of the minimizers $w_{\mu_n}$ to $w_\mu$ is equivalent to the $\gamma$-convergence of Definition \ref{gamma}.

\begin{oss}\label{convres}
An important well known characterization of the $\gamma$-convergence is the following: a sequence $\mu_n$ $\gamma$-converges to $\mu$, if and only if, the sequence of resolvent operators $\mathcal{R}_{\mu_n}$ associated to $-\Delta+\mu_n$, converges (in the strong convergence of linear operators on $L^2$) to the resolvent $\mathcal{R}_\mu$ of the operator $-\Delta +\mu$. A consequence of this fact is that the spectrum of the operator $-\Delta+\mu_n$ converges (pointwise) to the one of $-\Delta+\mu$.
\end{oss}

\begin{oss}\label{s2r2}
The space $\M^+_{\cp}(\O)$ endowed with the $\gamma$-convergence is metrizable. If $\O$ is bounded, one may take $d_\gamma(\mu,\nu)=\|w_\mu-w_\nu\|_{L^2}$. Moreover, in this case, in \cite{dmmo87} it is proved that the space $\M^+_{\cp}(\O)$ endowed with the metric $d_\gamma$ is compact.
\end{oss}

\begin{prop}\label{wgamma}
Let $\Omega\subset\R^d$ and let $V_n\in L^1(\O)$ be a sequence weakly converging in $L^1(\O)$ to a function $V$. Then the capacitary measures $V_n\,dx$ $\gamma$-converge to $V\,dx$.
\end{prop}

\begin{proof}
We have to prove that the solutions $u_n=R_{V_n}(1)$ of
$$\begin{cases}
-\Delta u_n+V_n(x)u_n=1\\
u\in H^1_0(\O)
\end{cases}$$
weakly converge in $H^1_0(\O)$ to the solution $u=R_V(1)$ of
$$\begin{cases}
-\Delta u+V(x)u=1\\
u\in H^1_0(\O),
\end{cases}$$
or equivalently that the functionals
$$J_n(u)=\int_\O|\nabla u|^2\,dx+\int_\O V_n(x)u^2\,dx$$
$\Gamma\big(L^2(\O)\big)$-converge to the functional
$$J(u)=\int_\O|\nabla u|^2\,dx+\int_\O V(x)u^2\,dx.$$
The $\Gamma$-liminf inequality (Definition \ref{Gamma} (a)) is immediate since, if $u_n\to u$ in $L^2(\O)$, we have
$$\int_\O|\nabla u|^2\,dx\le\liminf_{n\to\infty}\int_\O|\nabla u_n|^2\,dx$$
by the lower semicontinuity of the $H^1(\O)$ norm with respect to the $L^2(\O)$-convergence, and
$$\int_\O V(x)u^2\,dx\le\liminf_{n\to\infty}\int_\O V_n(x)u_n^2\,dx$$
by the strong-weak lower semicontinuity theorem for integral functionals (see for instance \cite{busc}).

Let us now prove the $\Gamma$-limsup inequality (Definition \ref{Gamma} (b)) which consists, given $u\in H^1_0(\O)$, in constructing a sequence $u_n\to u$ in $L^2(\O)$ such that
\be\label{potgls}
\limsup_{n\to\infty}\int_\O|\nabla u_n|^2\,dx+\int_\O V_n(x)u_n^2\,dx
\le\int_\O|\nabla u|^2\,dx+\int_\O V(x)u^2\,dx.
\ee
For every $t>0$ let $u^t=(u\wedge t)\vee(-t)$; then, by the weak convergence of $V_n$, for $t$ fixed we have
$$\lim_{n\to\infty}\int_\O V_n(x)|u^t|^2\,dx=\int_\O V(x)|u^t|^2\,dx,$$
and
$$\lim_{t\to+\infty}\int_\O V(x)|u^t|^2\,dx=\int_\O V(x)|u|^2\,dx.$$
Then, by a diagonal argument, we can find a sequence $t_n\to+\infty$ such that
$$\lim_{n\to\infty}\int_\O V_n(x)|u^{t_n}|^2\,dx=\int_\O V(x)|u|^2\,dx.$$
Taking now $u_n=u^{t_n}$, and noticing that for every $t>0$
$$\int_\O|\nabla u^t|^2\,dx\le\int_\O|\nabla u|^2\,dx,$$
we obtain (\ref{potgls}) and so the proof is complete.
\end{proof}

In the case of weak* convergence of measures the statement of Proposition \ref{wgamma} is no longer true, as the following proposition shows.

\begin{prop}\label{VgeW}
Let $\O\subset\R^d$ ($d\ge2$) be a bounded open set and let $V,W\in L^1_+(\O)$ be two functions such that $V\ge W$. Then, there is a sequence $V_n\in L^1_+(\O)$, uniformly bounded in $L^1(\O)$, such that the sequence of measures $V_n(x)\,dx$ converges weakly* to $V(x)\,dx$ and $\gamma$-converges to $W(x)\,dx$.
\end{prop}

\begin{proof}\footnote{the idea of this proof was suggested by Dorin Bucur}
Without loss of generality we can suppose $\int_\O(V-W)\,dx=1$. Let $\mu_n$ be a sequence of probability measures on $\O$ weakly* converging to $(V-W)\,dx$ and such that each $\mu_n$ is a finite sum of Dirac masses. For each $n\in\N$ consider a sequence of positive functions $V_{n,m}\in L^1(\O)$ such that $\int_\O V_{n,m}\,dx=1$ and $V_{n,m}dx$ converges weakly* to $\mu_n$ as $m\to\infty$. Moreover, we choose $V_{n,m}$ as a convex combination of functions of the form $|B_{1/m}|^{-1}\chi_{B_{1/m}(x_j)}$.

We now prove that for fixed $n\in\N$, $(V_{n,m}+W)\,dx$ $\gamma$-converges, as $m\to\infty$, to $W\,dx$ or, equivalently, that the sequence $w_{W+V_{n,m}}$ converges in $L^2$ to $w_W$, as $m\to\infty$. Indeed, by the weak maximum principle, we have
$$w_{W+I_{\O_{m,n}}}\le w_{W+V_{n,m}}\le w_W,$$
where $\O_{m,n}=\O\setminus \cup_j B_{1/m}(x_j)$ and $I_{\O_{m,n}}$ is as in \eqref{Iomega}.

Since a point has zero capacity in $\R^d$ ($d\ge2$) there exists a sequence $\phi_m\to0$ strongly in $H^1(\R^d)$ with $\phi_m=1$ on $B_{1/m}(0)$ and $\phi_m=0$ outside $B_{1/\sqrt m}(0)$. We have
\begin{align}
\int_\O|w_W-w_{W+I_{\O_{m,n}}}|^2\,dx
&\le2\|w_W\|_{L^\infty}\int_\O(w_W-w_{W+I_{\O_{m,n}}})\,dx\nonumber\\
&=4\|w_W\|_{L^\infty}\big(E(W+I_{\O_{m,n}})-E(W)\big)\label{VgeWeq1}\\
&\le4\|w_W\|_{L^\infty}\left(\int_\O\frac12|\nabla w_m|^2+\frac12 Ww_m^2-w_m\,dx\right.\nonumber\\
&\qquad\left.-\int_\O\frac12|\nabla w_W|^2+\frac12 Ww_W^2-w_W\,dx\right),\nonumber
\end{align}
where $w_m$ is any function in $\in H^1_0(\O_{m,n})$. Taking
$$w_m(x)=w_W(x)\prod_j(1-\phi_{m}(x-x_j)),$$
since $\phi_m\to0$ strongly in $H^1(\R^d)$, it is easy to see that $w_m\to w_W$ strongly in $H^1(\O)$ and so, by \eqref{VgeWeq1}, $w_{W+I_{\O_{m,n}}}\to w_W$ in $L^2(\O)$ as $m\to\infty$. Since the weak convergence of probability measures and the $\gamma$-convergence are both induced by metrics, a diagonal sequence argument brings to the conclusion. 
\end{proof}

\begin{oss}
When $d=1$, a result analogous to Proposition \ref{wgamma} is that any sequence $(\mu_n)$ weakly* converging to $\mu$ is also $\gamma$-converging to $\mu$. This is an easy consequence of the compact embedding of $H^1_0(\O)$ into the space of continuous functions on $\O$.
\end{oss}

We note that the hypothesis $V\ge W$ in Proposition \ref{VgeW} is necessary. Indeed, we have the following proposition, whose proof is contained in \cite[Theorem 3.1]{buvazo} and we report it here for the sake of completeness.

\begin{prop}\label{gamma<weak}
Let $\mu_n\in\M^+_{cap}(\O)$ be a sequence of capacitary Radon measures weakly* converging to the measure $\nu$ and $\gamma$-converging to the capacitary measure $\mu\in\M^+_{cap}(\O)$. Then $\mu\le\nu$ in $\O$.
\end{prop}

\begin{proof}
We note that it is enough to show that $\mu(K)\le\nu(K)$ whenever $K\subset\subset\O$ is a compact set. Let $u$ be a nonnegative smooth function with compact support in $\O$ such that $u\le 1$ in $\O$ and $u=1$ on $K$; we have
$$\mu(K)\le\int_\O u^2\,d\mu\le\liminf_{n\to\infty}\int_\O u^2\,d\mu_n=\int_\O u^2\,d\nu\le \nu\left(\{u>0\}\right).$$
Since $u$ is arbitrary, we have the conclusion by the Borel regularity of $\nu$.
\end{proof}

\section{Existence of optimal potentials in $L^p(\O)$}\label{s3}

In this section we consider the optimization problem 
\be\label{pop}
\min\left\{F(V)\ :\ V:\O\to [0,+\infty],\ \int_\O V^p\,dx\le1\right\},
\ee
where $p>0$ and $F(V)$ is a cost functional depending on the solution of some partial differential equation on $\O$. Typically, $F(V)$ is the minimum of some functional $J_V:H^1_0(\O)\to\R$ depending on $V$. A natural assumption in this case is the lower semicontinuity of the functional $F$ with respect to the $\gamma$-convergence, that is
\be\label{lsc}
F(\mu)\le\liminf_{n\to\infty}F(\mu_n),\qquad\hbox{whenever }\mu_n\to_\gamma\mu.
\ee

\begin{teo}\label{main}
Let $F:L^1_+(\O)\to\R$ be a functional, lower semicontinuous with respect to the $\gamma$-convergence, and let $\V$ be a weakly $L^1(\O)$ compact set. Then the problem 
\be\label{popK}
\min\left\{F(V)\ :\ V\in\V\right\},
\ee
admits a solution.
\end{teo}

\begin{proof}
Let $(V_n)$ be a minimizing sequence in $\V$. By the compactness assumption on $\V$, we may assume that $V_n$ tends weakly $L^1(\O)$ to some $V\in\V$. By Proposition \ref{wgamma}, we have that $V_n$ $\gamma$-converges to $V$ and so, by the semicontinuity of $F$,
$$F(V)\le\liminf_{n\to\infty}F(V_n),$$
which gives the conclusion.
\end{proof}

\begin{oss}
Theorem \ref{main} applies for instance to the integral functionals and to the spectral functionals considered in the introduction; it is not difficult to show that they are lower semicontinuous with respect to the $\gamma$-convergence.
\end{oss}

\begin{oss}
In some special cases the solution of \eqref{pop} can be written explicitly in terms of the solution of some partial differential equation on $\O$. This is the case of the Dirichlet Energy, that we discuss in Subsection \ref{s31}, and of the first eigenvalue of the Dirichlet Laplacian $\lambda_1$ (see \cite[Chapter 8]{hen03}).
\end{oss}

The compactness assumption on the admissible class $\V$ for the weak $L^1(\O)$ convergence in Theorem \ref{main} is for instance satisfied if $\Omega$ has finite measure and $\V$ is a convex closed and bounded subset of $L^p(\Omega)$, with $p\ge 1$. In the case of measures an analogous result holds.

\begin{teo}\label{mainmeas}
Let $\O\subset\R^d$ be a bounded open set and let $F:\M^+_{\cp}(\O)\to\R$ be a functional lower semicontinuous with respect to the $\gamma$-convergence. Then the problem 
\be\label{popmeas}
\min\left\{F(\mu)\ :\ \mu\in \M^+_{\cp}(\O),\ \mu(\O)\le1\right\},
\ee
admits a solution.
\end{teo}

\begin{proof}
Let $(\mu_n)$ be a minimizing sequence. Then, up to a subsequence $\mu_n$ converges weakly* to some measure $\nu$ and $\gamma$-converges to some measure $\mu\in\M^+_{\cp}(\O)$. By Proposition \ref{gamma<weak}, we have that $\mu(\O)\le\nu(\O)\le1$ and so, $\mu$ is a solution of \eqref{popmeas}.
\end{proof}
 
The following example shows that the optimal solution of problem \eqref{popmeas} is not, in general, a function $V(x)$, even when the optimization criterion is the energy $\E_f$ introduced in \eqref{energy}. On the other hand, an explicit form for the optimal potential $V(x)$ will be provided in Proposition \ref{maxone} assuming that the right-hand side $f$ is in $L^2(\O)$.
 
\begin{exam}\label{examdelta}
Let $\O=(-1,1)$ and consider the functional 
$$F(\mu)=-\min\left\{\frac12\int_\O|u'|^2\,dx+\frac12\int_\O u^2\,d\mu-u(0)\ :\ u\in H^1_0(\O)\right\}.$$
Then, for any $\mu$ such that $\mu(\O)\le1$, we have 
\be\label{examdelta1}
F(\mu)\ge-\min\left\{\frac12\int_\O|u'|^2\,dx+\frac12\big(\sup_\O u\big)^2-u(0)\ :\ u\in H^1_0(\O),\ u\ge0\right\}.
\ee
By a symmetrization argument, the minimizer $u$ of the right-hand side of \eqref{examdelta1} is radially decreasing; moreover, $u$ is linear on the set $u<M$, where $M=\sup u$, and so it is of the form 
\be
u(x)=\begin{cases}
\begin{array}{ll}
\frac{M}{1-\alpha}x+\frac{M}{1-\alpha},&\ x\in[-1,-\alpha],\\
M,&\ x\in[-\alpha,\alpha],\\
-\frac{M}{1-\alpha}x+\frac{M}{1-\alpha},&\ x\in[\alpha,1],
\end{array}
\end{cases}
\ee 
for some $\alpha\in[0,1]$. A straightforward computation gives $\alpha=0$ and $M=1/3$. Thus, $u$ is also the minimizer of 
$$F(\delta_0)=-\min\left\{\frac12\int_\O|u'|^2\,dx+\frac12u(0)^2-u(0)\ :\ u\in H^1_0(\O)\right\},$$
and so $\delta_0$ is the solution of
$$\min\left\{F(\mu)\ :\ \mu(\O)\le1\right\}.$$
\end{exam}

\subsection{Minimization problems in $L^p$ concerning the Dirichlet Energy functional}\label{s31}

Let $\O\subset\R^d$ be a bounded open set and let $f\in L^2(\O)$. By Theorem \ref{main}, the problem
\be\label{maxpb}
\min\left\{-\E_f(V)\ :\ V\in\V\right\}\qquad\hbox{with}\qquad\V=\left\{V\ge0,\ \int_\O V^p\,dx\le1\right\},
\ee
admits a solution, where $\E_f(V)$ is the energy functional defined in \eqref{energy}. We notice that, replacing $-\E_f(V)$ by $\E_f(V)$, makes problem \eqref{maxpb} trivial, with the only solution $V\equiv0$. Minimization problems for $\E_f$ will be considered in Section \ref{s4} for admissible classes of the form
$$\V=\left\{V\ge0,\ \int_\O V^{-p}\,dx\le1\right\}.$$
Analogous results for $F(V)=-\lambda_1(V)$ were proved in \cite[Theorem 8.2.3]{hen03}.

\begin{prop}\label{maxex}
Let $\O\subset\R^d$ be a bounded open set, $1<p<\infty$ and $f\in L^2(\O)$. Then the problem \eqref{maxpb} has a unique solution
$$V_p=\left(\int_\O|u_p|^{2p/(p-1)}\,dx\right)^{-1/p}|u_p|^{-1+(p+1)/(p-1)},$$
where $u_p\in H^1_0(\O)\cap L^{2p/(p-1)}(\O)$ is the minimizer of the functional
\be\label{Ja}
J_p(u):=\frac12 \int_\O|\nabla u|^2\,dx+\frac12\left(\int_\O|u|^{2p/(p-1)}\,dx\right)^{(p-1)/p}-\int_\O uf\,dx.
\ee
Moreover, we have $\E_f(V_p)=J_p(u_p)$.
\end{prop}

\begin{proof}
We first show that we have 
\be\label{maxexfeb1}
\max_{V\in\V}\ \min_{u\in H^1_0(\O)}\int_\O\left(\frac12|\nabla u|^2+u^2V-uf\right)\,dx\le\min_{u\in H^1_0(\O)}\ \max_{V\in\V}\int_\O\left(\frac12|\nabla u|^2+u^2V-uf\right)\,dx,
\ee
where the maximums are taken over all positive functions $V\in L^p(\O)$. For a fixed $u\in H^1_0(\O)$, the maximum on the right-hand side (if finite) is achieved for a function $V$ such that $\Lambda pV^{p-1}=u^2$, where $\Lambda$ is a Lagrange multiplier. By the condition $\int_\O V^p\,dx=1$ we obtain that the maximum is achieved for
$$V=\left(\int_\O|u|^{\frac{2p}{p-1}}\,dx\right)^{1/p}|u|^{\frac{2}{p-1}}.$$
Substituting in \eqref{maxexfeb1}, we obtain
\be\label{E<J}
\max\left\{\E_f(V)\ :\ V\in\V\right\}\le\min\left\{J_p(u):\ {u\in H^1_0(\O)}\right\}.
\ee
Let $u_n$ be a minimizing sequence for $J_p$. Since $\inf\,J_p\le0$, we can assume $J_p(u_n)\le0$ for each $n\in\N$. Thus, we have
\be\label{apriori}
\frac12 \int_\O|\nabla u_n|^2\,dx+\frac12\left(\int_\O|u_n|^{2p/(p-1)}\,dx\right)^{(p-1)/p}\le\int_\O u_nf\,dx\le C\|f\|_{L^2(\O)}\|\nabla u_n\|_{L^{2}},
\ee
where $C$ is a constant depending on $\O$. Thus we obtain
\be\label{apriori2}
\int_\O|\nabla u_n|^2\,dx+\left(\int_\O|u_n|^{2p/(p-1)}\,dx\right)^{(p-1)/p}\le 4C^2\|f\|^2_{L^2(\O)},
\ee
and so, up to subsequence $u_n$ converges weakly in $H^1_0(\O)$ and $L^{2p/(p-1)}(\O)$ to some $u_p\in H^1_0(\O)\cap L^{2p/(p-1)}(\O)$. By the semicontinuity of the $L^2$-norm of the gradient and the $L^{\frac{2p}{p-1}}$-norm and the fact that $\int_\O fu_n\,dx\to\int_\O fu_p\,dx$, as $n\to\infty$, we have that $u_p$ is a minimizer of $J_p$. By the strict convexity of $J_p$, we have that $u_p$ is unique. Moreover, by \eqref{apriori} and \eqref{apriori2}, $J_p(u_p)>-\infty$. Writing down the Euler-Lagrange equation for $u_p$, we obtain
$$-\Delta u_p+\left(\int_\O|u_p|^{2p/(p-1)}\,dx\right)^{-1/p}|u_p|^{2/(p-1)}u_p=f.$$
Setting
$$V_p=\left(\int_\O|u_p|^{2p/(p-1)}\,dx\right)^{-1/p}|u_p|^{2/(p-1)},$$
we have that $\int_\O V_p^p\,dx=1$ and $u_p$ is the solution of 
\be\label{eqalpha}
-\Delta u_p+V_p u_p=f.
\ee
In particular, we have $J_p(u_p)=\E_p(V_p)$ and so $V_p$ solves \eqref{maxpb}. The uniqueness of $V_p$ follows by the uniqueness of $u_p$ and the equality case in the H\"older inequality
$$\int_\O u^2 V\,dx\le\left(\int_\O V^p\,dx\right)^{1/p}\left(\int_\O|u|^{2p/(p-1)}\,dx\right)^{(p-1)/p}\le\left(\int_\O|u|^{2p/(p-1)}\,dx\right)^{(p-1)/p}.$$
\end{proof}

When the functional $F$ is the energy $\E_f$, the existence result holds also in the case $p=1$. Before we give the proof of this fact in Proposition \ref{maxone}, we need some preliminary results. We also note that the analogous results were obtained in the case $F=-\lambda_1$ (see \cite[Theorem 8.2.4]{hen03}) and in the case $F=-\E_f$, where $f$ is a positive function (see \cite{buvazo}).

\begin{oss}\label{unifest}
Let $u_p$ be the minimizer of $J_p$, defined in \eqref{Ja}. By \eqref{apriori2}, we have the estimate
\be\label{unifest1}
\|\nabla u_p\|_{L^2(\O)}+\|u_p\|_{L^{2p/(p-1)}(\O)}\le 4C^2\|f\|_{L^2(\O)},
\ee
where $C$ is the constant from \eqref{apriori}. Moreover, we have $u_p\in H^2_{loc}(\O)$ and for each open set $\O'\subset\subset\O$, there is a constant $C$ not depending on $p$ such that
$$\|u_p\|_{H^2(\O')}\le C(f,\O').$$
Indeed, $u_p$ satisfies the PDE
\be\label{unifest2}
-\Delta u+c|u|^\alpha u=f,
\ee
with $c>0$ and $\alpha=2/(p-1)$, and standard elliptic regularity arguments (see \cite[Section 6.3]{ev}) give that $u\in H^2_{loc}(\O)$. To show that $\|u_p\|_{H^2(\O')}$ is bounded independently of $p$ we apply the Nirenberg operator $\partial^h_ku=\frac{u(x+he_k)-u(x)}{h}$ on both sides of \eqref{unifest2}, and multiplying by $\phi^2\partial_k^h u$, where $\phi$ is an appropriate cut-off function which equals $1$ on $\O'$, we have 
\begin{align}\label{unifest3}
&\int_\O\phi^2|\nabla \partial_k^hu|^2\,dx+\int_\O\nabla(\partial_k^h u)\cdot\nabla(\phi^2)\partial_k^h u\,dx+c(\alpha+1)\int_\O\phi^2|u|^\alpha|\partial_k^h u|^2\,dx\\
&\hskip9truecm=-\int f\partial_k^h(\phi^2\partial_{k}^h u)\,dx,\nonumber
\end{align}
for all $k=1,\dots,d$. Some straightforward manipulations now give
\be\label{unifest4}
\|\nabla^2u\|_{L^2(\O')}^2\le \sum_{k=1}^d\int_\O\phi^2|\nabla \partial_ku|^2\,dx\le C(\O')\left(\|f\|_{L^2(\{\phi^2> 0\})}+\|\nabla u\|_{L^2(\O)}\right).
\ee
\end{oss}

\begin{lemma}\label{sc}
Let $\O\subset\R^d$ be an open set and $f\in L^2(\O)$. Consider the functional $J_1:L^{2}(\O)\to\R$ defined by
\be\label{J1}
J_1(u):=\frac12\int_\O|\nabla u|^2\,dx+\frac12\|u\|_\infty^2-\int_\O uf\,dx,
\ee
Then, $J_p$ $\Gamma$-converges in $L^{2}(\O)$ to $J_1$, as $p\to1$, where $J_p$ is defined in \eqref{Ja}. 
\end{lemma}
\begin{proof}
Let $v_n\in L^{2}(\O)$ be a sequence of positive functions converging in $L^{2}$ to $v\in L^{2}(\O)$ and let $\alpha_n\to+\infty$. Then, we have that 
\be\label{sc1}
\|v\|_{L^\infty(\O)}\le\liminf_{n\to\infty}\|v_n\|_{L^{\alpha_n}(\O)}.
\ee
In fact, suppose first that $\|v\|_{L^\infty}=M<+\infty$ and let $\omega_\eps=\{v>M-\eps\}$, for some $\eps>0$. Then, we have
$$\liminf_{n\to\infty}\|v_n\|_{L^{\alpha_n}(\O)}
\ge\lim_{n\to\infty}|\omega_\eps|^{(1-\alpha_n)/\alpha_n}\int_{\omega_\eps} v_n\,dx=|\omega_\eps|^{-1}\int_{\omega_\eps}v\,dx\ge M-\eps,$$
and so, letting $\eps\to0$, we have $\liminf_{n\to\infty}\|v_n\|_{L^{\alpha_n}(\O)}\le M$. If $\|v\|_{L^\infty}=+\infty$, then setting $\omega_k=\{v>k\}$, for any $k\ge1$, and arguing as above, we obtain \eqref{sc1}.\\
Let $u_n\to u$ in $L^2(\O)$. Then, by the semicontinuity of the $L^2$ norm of the gradient and \eqref{sc1} and the continuity of the term $\int_\O uf\,dx$, we have 
\be
J_1(u)\le\liminf_{n\to\infty}J_{p_n}(u_n),
\ee 
for any decreasing sequence $p_n\to1$. On the other hand, for any $u\in L^2$, we have $J_{p_n}(u)\to J_1(u)$ as $n\to\infty$ and so, we have the conclusion.
\end{proof}

\begin{prop}\label{maxone}
Let $\O\subset\R^d$ be a bounded open set and $f\in L^2(\O)$. Then there is a unique solution of problem \eqref{maxpb} with $p=1$, given by
$$V_1=\frac1M \left(\chi_{\omega_+} f-\chi_{\omega_-}f\right),$$
where $M=\|u_1\|_{L^\infty(\O)}$, $\omega_+=\{u_1=M\}$, $\omega_-=\{u_1=-M\}$, being $u_1\in H^1_0(\O)\cap L^{\infty}(\O)$ the unique minimizer of the functional $J_1$, defined in \eqref{J1}. In particular, $\int_{\omega_+}f\,dx-\int_{\omega_-}f\,dx=M$, $f\ge0$ on $\omega_+$ and $f\le0$ on $\omega_-$.
\end{prop}

\begin{proof}
For any $u\in H^1_0(\O)$ and any $V\ge0$ with $\int_\O V\,dx\le1$ we have
$$\int_\O u^2V\,dx\le\|u\|_\infty^2\int_\O V\,dx\le \|u\|_\infty^2,$$
where for sake of simplicity, we write $\|\cdot\|_\infty$ instead of $\|\cdot\|_{L^\infty(\O)}$. Arguing as in the proof of Proposition \ref{maxex}, we obtain the inequalities
\begin{align*}
&\frac12\int_\O|\nabla u|^2\,dx+\frac12\int_\O u^2V\,dx-\int_\O uf\,dx\,\le J_1(u),\\
&\max\left\{\E_f(V)\ :\ \int_\O V\le1\right\}\le\min\left\{J_1(u)\ :\ u\in H^1_0(\O)\right\}.
\end{align*}
As in \eqref{apriori}, we have that a minimizing sequence of $J_1$ is bounded in $H^1_0(\O)\cap L^\infty(\O)$ and thus by semicontinuity there is a minimizer $u_1\in H^1_0(\O)\cap L^\infty(\O)$ of $J_1$, which is also unique, by the strict convexity of $J_1$. Let $u_p$ denotes the minimizer of $J_p$ as in Proposition \ref{maxex}. Then, by Remark \ref{unifest}, we have that the family $u_p$ is bounded in $H^1_0(\O)$ and in $H^2(\O')$ for each $\O'\subset\subset\O$. Then, we have that each sequence $u_{p_n}$ has a subsequence converging weakly in $L^2(\O)$ to some $u\in H^2_{loc}(\O)\cap H^1_0(\O)$. By Lemma \ref{sc}, we have $u=u_1$ and so, $u_1\in H^2_{loc}(\O)\cap H^1_0(\O)$. Thus $u_{p_n}\to u_1$ in $L^2(\O)$.

Let us define $M=\|u_1\|_\infty$ and $\omega=\omega_+\cup\omega_-$. We claim that $u_1$ satisfies, on $\O$ the PDE
\be\label{maxone1}
-\Delta u+\chi_\omega f=f.
\ee
Indeed, setting $\O_t=\O\cap\{|u|<t\}$ for $t>0$, we compute the variation of $J_1$ with respect to any function $\varphi\in H^1_0(\O_{M-\eps})$. Namely we consider functions of the form $\varphi=\psi w_\eps$ where $w_\eps$ is the solution of $-\Delta w_\eps=1$ on $\O_{M-\eps}$, and $w_\eps=0$ on $\partial\O_{M-\eps}$. Thus we obtain that $-\Delta u_1=f$ on $\O_{M-\eps}$ and letting $\eps\to0$ we conclude, thanks to the Monotone Convergence Theorem, that
$$-\Delta u_1=f\qquad\hbox{on }\O_M=\O\setminus\omega.$$
Moreover, since $u_1\in H^2_{loc}(\O)$, we have that $\Delta u_1=0$ on $\omega$ and so, we obtain \eqref{maxone1}.\\
Since $u_1$ is the minimizer of $J_1$, we have that for each $\eps\in\R$, $J_1((1+\eps)u_1)-J_1(u_1)\ge0$. Taking the derivative of this difference at $\eps=0$, we obtain
\be\label{maxone2}
\int_\O|\nabla u_1|^2\,dx+M^2=\int_\O fu_1\,dx.
\ee
By \eqref{maxone1}, we have $\int_\O|\nabla u_1|^2\,dx=\int_{\O\setminus\omega}fu_1\,dx$ and so
\be\label{maxone2}
M=\int_{\omega_+}f\,dx-\int_{\omega_-}f\,dx.
\ee
Setting $V_1:=\frac1M \left(\chi_{\omega_+} f-\chi_{\omega_-}f\right)$, we have that $\int_\O V_1\,dx=1$, $-\Delta u_1+V_1u_1=f$ in $H^{-1}(\O)$ and 
$$J_1(u_1)=\frac12\int_\O|\nabla u_1|^2\,dx+\frac12\int_\O u_1^2V_1\,dx-\int_\O u_1f\,dx.$$
We are left to prove that $V_1$ is admissible, i.e. $V_1\ge0$. To do this, consider $w_\eps$ the energy function of the quasi-open set $\{u<M-\eps\}$ and let $\varphi=w_\eps\psi$ where $\psi\in C^\infty_c(\R^d)$, $\psi\ge0$. Since $\varphi\ge0$, we get that
$$0\le\lim_{t\to0^+}\frac{J_1(u_1+t\varphi)-J_1(u_1)}{t}=\int_\O\langle\nabla u_1,\nabla\varphi\rangle\,dx-\int_\O f\varphi\,dx.$$
This inequality holds for any $\psi$ so that, integrating by parts, we obtain
$$-\Delta u_1-f\ge0$$
almost everywhere on $\{u_1<M-\eps\}$. In particular, since $\Delta u_1=0$ almost everywhere on $\omega_-=\{u=-M\}$, we obtain that $f\le0$ on $\omega_-$. Arguing in the same way, and considering test functions supported on $\{u_1\ge-M+\eps\}$, we can prove that $f\ge0$ on $\omega_+$. This implies $V_1\ge0$ as required.
\end{proof}

\begin{oss}
Under some additional assumptions on $\O$ and $f$ one can obtain some more precise regularity results for $u_1$. In fact, in \cite[Theorem A1]{egnell} it was proved that if $\partial\O\in C^2$ and if $f\in L^\infty(\O)$ is positive, then $u_1\in C^{1,1}(\overline\O)$.
\end{oss}

\begin{oss}\label{controesempio}
In the case $p<1$ problem \eqref{maxpb} does not admit, in general, a solution, even for regular $f$ and $\O$. We give a counterexample in dimension one, which can be easily adapted to higher dimensions. 

Let $\O=(0,1)$, $f=1$, and let $x_{n,k}=k/n$ for any $n\in\N$ and $k=1,\dots,n-1$. We define the (capacitary) measures
$$\mu_n=\sum_{k=1}^{n-1} +\infty\, \delta_{x_{n,k}},$$
where $\delta_x$ is the Dirac measure at the point $x$. Let $w_n$ be the minimizer of the functional $J_{\mu_n}(1,\cdot)$, defined in \eqref{F}. Then $w_n$ vanishes at $x_{n,k}$, for $k=1,\dots,n-1$, and so we have
$$\E(\mu_n)=n\,\min\left\{\frac12\int_0^{1/n}|u'|^2\,dx-\int_0^{1/n}u\,dx\ :\ u\in H^1_0(0,1/n)\right\}=-\frac{C}{n^2},$$
where $C>0$ is a constant.

For any fixed $n$ and $j$, let $V_j^n$ be the sequence of positive functions such that $\int_0^1|V_j^n|^p\,dx=1$, defined by
\be\label{esempioVjn}
V_j^n=C_n\sum_{k=1}^{n-1}j^{1/p}\chi_{\left[\frac{k}{n}-\frac1j,\frac{k}{n}+\frac1j\right]}<\sum_{k=1}^{n-1}I_{\left[\frac{k}{n}-\frac1j,\frac{k}{n}+\frac1j\right]},
\ee
where $C_n$ is a constant depending on $n$ and $I$ is as in \eqref{Iomega}.
By the compactness of the $\gamma$-convergence, we have that, up to a subsequence, $V_j^n\,dx$ $\gamma$-converges to some capacitary measure $\mu$ as $j\to\infty$. On the other hand it is easy to check that $\sum_{k=1}^{n-1}I_{\left[\frac{k}{n}-\frac1j,\frac{k}{n}+\frac1j\right]}(x)$ $\gamma$-converges to $\mu_n$ as $j\to\infty$. By \eqref{esempioVjn}, we have that $\mu\le\mu_n$. In order to show that $\mu=\mu_n$ it is enough to check that each nonnegative function $u\in H^1_0((0,1))$, for which $\int u^2\,d\mu<+\infty$, vanishes at $x_{n,k}$ for $k=1,\dots,n-1$. Suppose that $u(k/n)>0$. By the definition of the $\gamma$-convergence, there is a sequence $u_j\in H^1_0(\O)=H^1_{V_j^n}(\O)$ such that $u_j\to u$ weakly in $H^1_0(\O)$ and $\int u_j^2 V_j^n\,dx\le C$, for some constant $C$ not depending on $j\in\N$. Since $u_j$ are uniformly $1/2$-H\"older continuous, we can suppose that $u_j\ge\eps>0$ on some interval $I$ containing $k/n$. But then for $j$ large enough $I$ contains $[k/n-1/j,k/n+1/j]$ so that
$$C\ge\int_0^1u_j^2V_j^n\,dx\ge\int_{k/n-1/j}^{k/n+1/j}u_j^2V_j^n\,dx\ge2C_n\eps^2j^{1/p-1},$$
which is a contradiction for $p<1$. Thus, we have that $\mu=\mu_n$ and so $V_j^n$ $\gamma$-converges to $\mu_n$ as $j\to\infty$. In particular, $\E(\mu_n)=\lim_{j\to\infty}\E_1(V^n_j)$ and since the left-hand side converges to zero as $n\to\infty$, we can choose a diagonal sequence $V^n_{j_n}$ such that $\E(V^n_{j_n})\to 0$ as $n\to\infty$. Since there is no admissible functional $V$ such that $\E_1(V)=0$, we have the conclusion.
\end{oss}

\section{Existence of optimal potentials for unbounded constraints}\label{s4}

In this section we consider the optimization problem
\be\label{popPhi}
\min\left\{F(V)\ :\ V\in\V\right\},
\ee
where $\V$ is an admissible class of nonnegative Borel functions on the bounded open set $\O\subset\R^d$ and $F$ is a cost functional on the family of capacitary measures $\M^+_{\cp}(\O)$. The admissible classes we study depend on a function $\Psi:[0,+\infty]\to[0,+\infty]$
$$\V=\left\{V:\O\to[0,+\infty]\ :\ V\hbox{ Lebesgue measurable, }\int_\O\Psi(V)\,dx\le1\right\}.$$

\begin{teo}\label{mainPhi}
Let $\O\subset\R^d$ be a bounded open set and $\Psi:[0,+\infty]\to[0,+\infty]$ a strictly decreasing function with $\Psi^{-1}$ convex. Then, for any functional $F:\M^+_{\cp}(\O)\to\R$ which is increasing and lower semicontinuous with respect to the $\gamma$-convergence, the problem \eqref{popPhi} has a solution. 
\end{teo}

\begin{proof}
Let $V_n\in\V$ be a minimizing sequence for problem \eqref{popPhi}. Then, $v_n:=\Psi(V_n)$ is a bounded sequence in $L^1(\O)$ and so, up to a subsequence, $v_n$ converges weakly* to some measure $\nu$. We will prove that $V:=\Psi^{-1}(\nu_a)$ is a solution of \eqref{popPhi}, where $\nu_a$ denotes the density of the absolutely continuous part of $\nu$ with respect to the Lebesgue measure. Clearly $V\in\V$ and so it remains to prove that $F(V)\le\liminf_n F(V_n)$. In view of Remark \ref{s2r2}, we can suppose that, up to a subsequence, $V_n$ $\gamma$-converges to a capacitary measure $\mu\in\M^+_{\cp}(\O)$. We claim that the following inequalities hold true: 
\be\label{th1}
F(V)\le F(\mu)\le\liminf_{n\to\infty} F(V_n).
\ee
In fact, the second inequality in \eqref{th1} is the lower semicontinuity of $F$ with respect to the $\gamma$-convergence, while the first needs a more careful examination. By the definition of $\gamma$-convergence, we have that for any $u\in H^1_0(\O)$, there is a sequence $u_n\in H^1_0(\O)$ which converges to $u$ in $L^2(\O)$ and is such that
\begin{align}
\int_\O|\nabla u|^2dx+\int_\O u^2d\mu
&=\lim_{n\to\infty}\int_\O|\nabla u_n|^2dx+\int_\O u_n^2 V_n\,dx\nonumber\\
&=\lim_{n\to\infty}\int_\O|\nabla u_n|^2dx+\int_\O u_n^2\Psi^{-1}(v_n)\,dx\label{ineqth2}\\
&\ge\int_\O|\nabla u|^2dx+\int_\O u^2\Psi^{-1}(\nu_a)\,dx\nonumber\\
&=\int_\O|\nabla u|^2dx+\int_\O u^2V\,dx,\nonumber
\end{align}
where the inequality in \eqref{ineqth2} is due to strong-weak* lower semicontinuity of integral functionals (see for instance \cite{busc}). Thus, for any $u\in H^1_0(\O)$, we have 
$$\int_\O u^2\,d\mu\ge\int_\O u^2V\,dx,$$
which gives $V\le\mu$. Since $F$ is increasing, we obtain the first inequality in \eqref{th1} and so the conclusion.
\end{proof}

\begin{oss}
The condition on the function $\Psi$ in Theorem \ref{mainPhi} is satisfied for instance by the following functions:
\begin{enumerate}
\item $\Psi(x)=x^{-p}$, for any $p>0$;
\item $\Psi(x)=e^{-\alpha x}$, for any $\alpha>0$.
\end{enumerate}
\end{oss}

\subsection{Optimal potentials for the Dirichlet Energy and the first eigenvalue of the Dirichlet Laplacian}\label{s41}

In some special cases, the solution of the optimization problem \eqref{popPhi} can be computed explicitly through the solution of some PDE, as in Subsection \ref{s31}. This occurs for instance when $F=\lambda_1$ or when $F=\E_f$, with $f\in L^2(\O)$. We note that, by the variational formulation
\be
\lambda_1(V)=\min\left\{\int_\O|\nabla u|^2dx+\int_\O u^2V\,dx\ :\ u\in H^1_0(\O),\ \int_\O u^2\,dx=1\right\},
\ee
we can rewrite problem \eqref{popPhi} as
\be\label{op3}
\begin{array}{ll}
\ds\min\left\{\min_{\|u\|_2=1}\Big\{\int_\O|\nabla u|^2\,dx+\int_\O u^2V\,dx\Big\}\ :\ V\ge0,\ \int_\O\Psi(V)\,dx\le1\right\}\\
\ds\qquad\qquad=\min_{\|u\|_2=1}\left\{\min\Big\{\int_\O|\nabla u|^2\,dx+\int_\O u^2V\,dx\ :\ V\ge0,\ \int_\O\Psi(V)\,dx\le1\Big\}\right\}.
\end{array}
\ee
One can compute that, if $\Psi$ is differentiable with $\Psi'$ invertible, then  the second minimum in \eqref{op3} is achieved for
\be\label{vopt}
V=(\Psi')^{-1}(\Lambda_u u^2),
\ee
where $\Lambda_u$ is a constant such that $\int_\O\Psi\left((\Psi')^{-1}(\Lambda_u u^2)\right)\,dx=1$. Thus, the solution of the problem on the right hand side of \eqref{op3} is given through the solution of
\be\label{gs1}
\min\left\{\int_\O|\nabla u|^2\,dx+\int_\O u^2 (\Psi')^{-1}(\Lambda_u u^2)\,dx:\ u\in H^1_0(\O),\ \int_\O u^2\,dx=1\right\}.
\ee
Analogously, we obtain that the optimal potential for the Dirichlet Energy $\E_f$ is given by \eqref{vopt}, where this time $u$ is a solution of 
\be\label{gs2}
\min\left\{\int_\O\frac12|\nabla u|^2\,dx+\int_\O\frac12u^2(\Psi')^{-1}(\Lambda_u u^2)\,dx-\int_\O fu\,dx\ :\ u\in H^1_0(\O)\right\}.
\ee 
Thus we obtain the following result.

\begin{cor}\label{lb}
Under the assumptions of Theorem \ref{mainPhi}, for the functionals $F=\lambda_1$ and $F=\E_f$ there exists a solution of \eqref{popPhi} given by $V=(\Psi')^{-1}(\Lambda_u u^2)$, where $u\in H^1_0(\O)$ is a minimizer of \eqref{gs1}, in the case $F=\lambda_1$, and of \eqref{gs2}, in the case $F=\E_f$.
\end{cor} 

\begin{exam}
If $\Psi(x)=x^{-p}$ with $p>0$, the optimal potentials for $\lambda_1$ and $\E_f$ are given by
\be\label{vopt2}
V=\left(\int_\O|u|^{2p/(p+1)}\,dx\right)^{1/p}u^{-2/(p+1)},
\ee 
where $u$ is the minimizer of \eqref{gs1} and \eqref{gs2}, respectively. We also note that, in this case
$$\int_\O u^2(\Psi')^{-1}(\Lambda_u u^2)\,dx=\left(\int_\O|u|^{2p/(p+1)}\,dx\right)^{(1+p)/p}.$$
\end{exam}

\begin{exam}
If $\Psi(x)=e^{-\alpha x}$ with $\alpha>0$, the optimal potentials for $\lambda_1$ and $\E_f$ are given by
\be\label{vopt2}
V=\frac{1}{\alpha}\left(\log\left(\int_\O u^2\,dx\right)-\log\left(u^2\right)\right),
\ee 
where $u$ is the minimizer of \eqref{gs1} and \eqref{gs2}, respectively.
We also note that, in this case
$$\int_\O u^2 (\Psi')^{-1}(\Lambda_u u^2)\,dx=\frac{1}{\alpha}\left(\int_\O u^2\,dx \int_\O\log\left(u^2\right)\,dx-\int_\O u^2\log\left(u^2\right)\,dx\right).$$
\end{exam}

\section{Optimization problems in unbounded domains}\label{s5}

In this section we consider optimization problems for which the domain region is the entire Euclidean space $\R^d$. General existence results, in the case when the design region $\Omega$ is unbounded, are hard to achieve since most of the cost functionals are not semicontinuous with respect to the $\gamma$-convergence in these domains. For example, it is not hard to check that if $\mu$ is a capacitary measure, infinite outside the unit ball $B_1$, then, for every $x_n\to\infty$, the sequence of translated measures $\mu_n=\mu(\cdot+x_n)$ $\gamma$-converges to the capacitary measure
$$I_\emptyset(E)=\begin{cases}
0,&\hbox{if }\cp(E)=0,\\
+\infty,&\hbox{if }\cp(E)>0.
\end{cases}$$
Thus increasing and translation invariant functionals are never lower semicontinuous with respect to the $\gamma$-convergence. In some special cases, as the Dirichlet Energy or the first eigenvalue of the Dirichlet Laplacian, one can obtain existence results by more direct methods, as those in Proposition \ref{maxex}.

For a potential $V\ge0$ and a function $f\in L^q(\R^d)$, we define the Dirichlet energy as
\be\label{energyrd}
\E_f(V)=\inf\left\{\int_{\R^d}\Big(\frac12|\nabla u|^2+\frac12V(x)u^2-f(x)u\Big)\,dx\ :\ u\in C^\infty_c(\R^d)\right\}.
\ee
In some cases it is convenient to work with the space $\dt{H}^1(\R^d)$, obtained as the closure of $C^\infty_c(\R^d)$ with respect to the $L^2$ norm of the gradient, instead of the classical Sobolev space $H^1(\R^d)$. We recall that if $d\ge 3$, the Gagliardo-Nirenberg-Sobolev inequality
\be\label{gnsd3}
\|u\|_{L^{2d/(d-2)}}\le C_d\|\nabla u\|_{L^2},\qquad\forall u\in \dt{H}^1(\R^d),
\ee
holds, while in the cases $d\le 2$, we have respectively
\begin{align}
&\|u\|_{L^\infty}\le\left(\frac{r+2}{2}\right)^{2/(r+2)}\|u\|_{L^r}^{r/(r+2)}\|u'\|_{L^2}^{2/(r+2)},\qquad\forall r\ge1,\ \forall u\in\dt{H}^1(\R);\label{gnsd1}\\
&\|u\|_{L^{r+2}}\le\left(\frac{r+2}{2}\right)^{2/(r+2)}\|u\|_{L^r}^{r/(r+2)}\|\nabla u\|_{L^2}^{2/(r+2)},\qquad\forall r\ge1,\ \forall u\in\dt{H}^1(\R^2).\label{gnsd2}
\end{align}

\subsection{Optimal potentials in $L^p(\R^d)$}

In this section we consider optimization problems for the Dirichlet energy $\E_f$ among potentials $V\ge0$ satisfying a constraint of the form $\|V\|_{L^p}\le1$. We note that the results in this section hold in a generic unbounded domain $\O$. Nevertheless, for sake of  simplicity, we restrict our attention to the case $\O=\R^d$.

\begin{prop}\label{pinterv}
Let $p>1$ and let $q$ be in the interval with end-points $a=2p/(p+1)$ and $b=\max\{1,2d/(d+2)\}$ (with $a$ included for every $d\ge1$, and $b$ included for every $d\ne2$). Then, for every $f\in L^q(\R^d)$, there is a unique solution of the problem
\be\label{maxrd}
\max\left\{\E_f(V)\ :\ V\ge0,\ \int_{\R^d}V^p\,dx\le1\right\}.
\ee 
\end{prop}

\begin{proof}
Arguing as in Proposition \ref{maxex}, we have that for $p>1$ the optimal potential $V_p$ is given by
\be\label{Vprd}
V_p=\left(\int_{\R^d}|u_p|^{2p/(p-1)}\,dx\right)^{-1/p}|u_p|^{2/(p-1)},
\ee
where $u_p$ is the solution of the problem
\begin{align}\label{Jard}
&\min\Bigg\{\frac12\int_{\R^d}|\nabla u|^2\,dx+\frac12\left(\int_{\R^d}|u|^{2p/(p-1)}\,dx\right)^{(p-1)/p}-\int_{\R^d}uf\,dx\ :\\
&\hskip7truecm u\in\dt{H}^1(\R^d)\cap L^{2p/(p-1)}(\R^d)\Bigg\}.\nonumber
\end{align}
Thus, it is enough to prove that there exists a solution of \eqref{Jard}. For a minimizing sequence $u_n$ we have
\be\label{apriorird}
\frac12\int_{\R^d}|\nabla u_n|^2\,dx+\frac12\left(\int_{\R^d}|u_n|^{2p/(p-1)}\,dx\right)^{(p-1)/p}\le\int_{\R^d}u_nf\,dx\le C\|f\|_{L^q}\|u_n\|_{L^{q'}}.
\ee
Suppose that $d\ge 3$. Interpolating $q'$ between $2p/(p-1)$ and $2d/(d-2)$ and using the Gagliardo-Nirenberg-Sobolev inequality \eqref{gnsd3}, we obtain that there is a constant $C$, depending only on $p,d$ and $f$, such that
$$\frac12\int_{\R^d}|\nabla u_n|^2\,dx+\frac12\left(\int_{\R^d}|u_n|^{2p/(p-1)}\,dx\right)^{(p-1)/p}\le C.$$
Thus we can suppose that $u_n$ converges weakly in $\dt{H}^1(\R^d)$ and in $L^{2p/(p-1)}(\R^d)$ and so, the problem \eqref{Jard} has a solution. In the case $d\le2$, the claim follows since, by using \eqref{gnsd1}, \eqref{gnsd2} and interpolation, we can still estimate $\|u_n\|_{L^{q'}}$ by means of $\|\nabla u_n\|_{L^2}$ and $\|u_n\|_{L^{2p/(p-1)}}$.
\end{proof}

Repeating the argument of Subsection \ref{s31}, one obtains an existence result for \eqref{maxrd} in the case $p=1$, too.

\begin{prop}\label{maxonerd}
Let $f\in L^q(\R^d)$, where $q\in[1,\frac{2d}{d+2}]$, if $d\ge 3$, and $q=1$, if $d=1,2$. Then there is a unique solution $V_1$ of problem \eqref{maxrd} with $p=1$, which is given by
$$V_1=\frac{f}{M}\left(\chi_{\omega_+}-\chi_{\omega_-}\right),$$
where $M=\|u_1\|_{L^\infty(\R^d)}$, $\omega_+=\{u_1=M\}$, $\omega_-=\{u_1=-M\}$, and $u_1$ is the unique minimizer of 
\be\label{J1xx}
\min\left\{\frac12\int_{\R^d}|\nabla u|^2\,dx+\frac12\|u\|_{L^\infty}^2-\int_{\R^d} uf\,dx:\ u\in\dt H^1(\R^d)\cap L^\infty(\R^d)\right\}.
\ee
In particular, $\int_{\omega_+} f\,dx-\int_{\omega_-}f\,dx=M$, $f\ge 0$ on $\omega_+$ and $f\le0$ on $\omega_-$.
\end{prop}

We note that, when $p=1$, the support of the optimal potential $V_1$ is contained in the support of the function $f$. This is not the case if $p>1$, as the following example shows.

\begin{exam}\label{exammaxrd}
Let $f=\chi_{B(0,1)}$ and $p>1$. By our previous analysis we know that there exist a solution $u_p$ for problem \eqref{Jard} and a solution $V_p$ for problem \eqref{maxrd} given by \eqref{Vprd}. We note that $u_p$ is positive, radially decreasing and satisfies the equation
$$-u''(r)-\frac{d-1}{r}u'(r)+C u^\alpha=0,\qquad r\in(1,+\infty),$$
where $\alpha=2p/(p-1)>2$ and $C$ is a positive constant. Thus, we have that 
$$u_p(r)=kr^{2/(1-\alpha)},$$
where $k$ is an explicit constant depending on $C$, $d$ and $\alpha$. In particular, we have that $u_p$ is not compactly supported on $\R^d$ (see Figure \ref{fig1}).

\begin{figure}[h]
\includegraphics[scale=1]{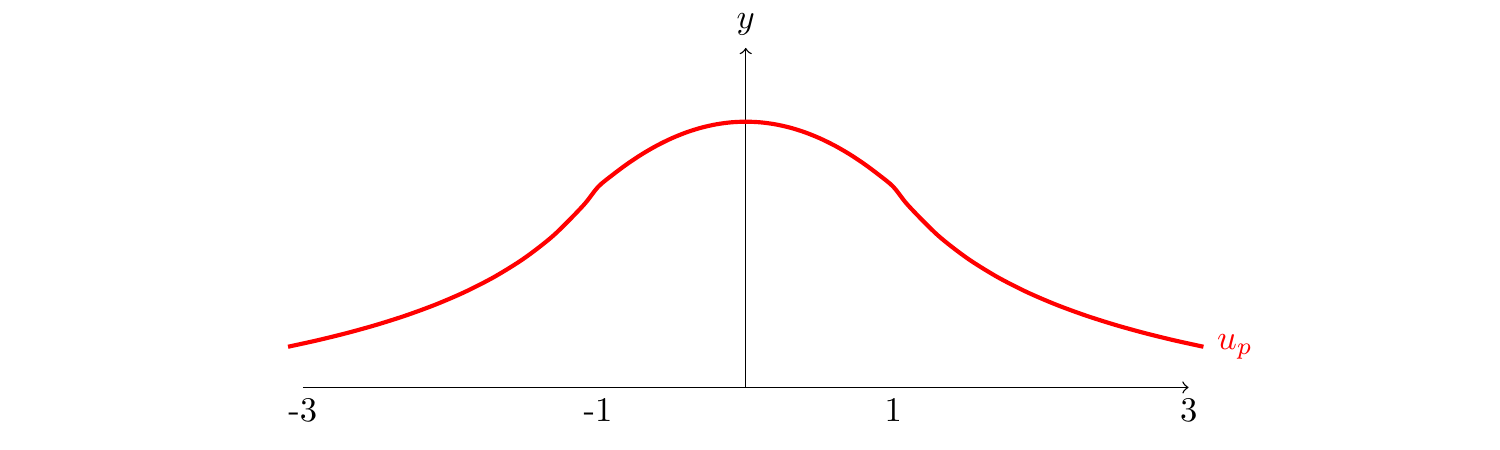}
\caption{The solution $u_p$ of problem \eqref{Jard}, with $p>1$ and $f=\chi_{B(0,1)}$ does not have a compact support.}\label{fig1}
\end{figure}

\end{exam}

\subsection{Optimal potentials with unbounded constraint}\label{ss52}

In this subsection we consider the problems 
\begin{align}
&\min\left\{\E_f(V)\ :\ V\ge0,\ \int_{\R^d}V^{-p}\,dx\le1\right\},\label{minrd}\\
&\min\left\{\lambda_1(V)\ :\ V\ge0,\ \int_{\R^d}V^{-p}\,dx\le1\right\},\label{lbrd}
\end{align}
for $p>0$ and $f\in L^q(\R^d)$. We will see in Proposition \ref{exErd} that in order to have existence for \eqref{minrd} the parameter $q$ must satisfy some constraint, depending on the value of $p$ and on the dimension $d$. Namely, we need $q$ to satisfy the following conditions
\begin{align}
q\in[\frac{2d}{d+2},\frac{2p}{p-1}],\ \hbox{if}\ d\ge3\ \hbox{and}\ p>1,\nonumber\\
q\in[\frac{2d}{d+2},+\infty],\ \hbox{if}\ d\ge3\ \hbox{and}\ p\le1,\nonumber\\
q\in(1,\frac{2p}{p-1}],\ \hbox{if}\ d=2\ \hbox{and}\ p>1,\label{admq}\\
q\in(1,+\infty],\  \hbox{if}\ d=2\ \hbox{and}\ p\le1,\nonumber\\
q\in[1,\frac{2p}{p-1}],\ \hbox{if}\ d=1\ \hbox{and}\ p>1,\nonumber\\
q\in[1,+\infty],\ \hbox{if}\ d=1\ \hbox{and}\ p\le1\nonumber.
\end{align}
We say that $q=q(p,d)\in[1,+\infty]$ is \emph{admissible} if it satisfy \eqref{admq}. Note that $q=2$ is admissible for any $d\ge1$ and any $p>0$. 

\begin{prop}\label{exErd}
Let $p>0$ and $f\in L^q(\R^d)$, where $q$ is admissible in the sense of \eqref{admq}. Then the minimization problem \eqref{minrd} has a solution $V_p$ given by 
\be\label{V-prd}
V_p=\left(\int_{\R^d} |u_p|^{2p/(p+1)}\,dx\right)^{1/p}|u_p|^{-2/(1+p)},
\ee
where $u_p$ is a minimizer of
\begin{align}\label{J-ard}
&\min\Bigg\{\frac12 \int_{\R^d} |\nabla u|^2\,dx+\frac12\left(\int_{\R^d} |u|^{2p/(p+1)}\,dx\right)^{(p+1)/p}-\int_{\R^d} uf\,dx:\\
&\hskip7truecm u\in \dt{H}^1(\R^d),\ |u|^{2p/(p+1)}\in L^1(\R^d)\Bigg\}.\nonumber
\end{align}
Moreover, if $p\ge1$, then the functional in \eqref{J-ard} is convex, its minimizer  is unique and so is the solution of \eqref{minrd}.
\end{prop}
\begin{proof}
By means of \eqref{gnsd3}, \eqref{gnsd1} and \eqref{gnsd2}, and thanks to the admissibility of $q$, we get the existence of a solution of \eqref{J-ard} through an interpolation argument similar to the one used in the proof of Proposition \ref{pinterv}. The existence of an optimal potential follows by the same argument as in Subsection \ref{s41}.
\end{proof}

In Example \ref{exammaxrd}, we showed that the optimal potentials for \eqref{maxrd}, may be supported on the whole $\R^d$. The analogous question for the problem \eqref{minrd} is whether the optimal potentials given by \eqref{V-prd} have a bounded set of finiteness $\{V_p<+\infty\}$. In order to answer this question, it is sufficient to study the support of the solutions $u_p$ of \eqref{J-ard}, which solve the equation
\be\label{eulagrd}
-\Delta u+C_p|u|^{-2/(p+1)}u=f,
\ee
where $C_p>0$ is a constant depending on $p$.

\begin{prop}\label{cs1}
Let $p>0$ and let $f\in L^q(\R^d)$, for $q>d/2$, be a nonnegative function with a compact support. Then every solution $u_p$ of problem \eqref{J-ard} has a compact support.
\end{prop}

\begin{proof}
With no loss of generality we may assume that $f$ is supported in the unit ball of $\R^d$. We first prove the result when $f$ is radially decreasing. In this case $u_p$ is also radially decreasing and nonnegative. Let $v$ be the function defined by $v(|x|)=u_p(x)$. Thus $v$ satisfies the equation
\be\label{ode1}
\begin{cases}
\ds-v''-\frac{d-1}{r}v'+C_p v^s=0\qquad r\in(1,+\infty),\\
v(1)=u_p(1),
\end{cases}
\ee
where $s=(p-1)/(p+1)$ and $C_p>0$ is a constant depending on $p$. Since $v\ge0$ and $v'\le 0$, we have that $v$ is convex. Moreover, since 
$$\int_1^{+\infty}v^2 r^{d-1}\,dr<+\infty,\qquad\int_1^{+\infty}|v'|^2 r^{d-1}\,dr<+\infty,$$ 
we have that $v$, $v'$ and $v''$ vanish at infinity. Multiplying \eqref{ode1} by $v'$ we obtain
$$\left(\frac{v'(r)^2}{2}-C_p\frac{v(r)^{s+1}}{s+1}\right)'=-\frac{d-1}{r}v'(r)^2\le0.$$
Thus the function $v'(r)^2/2-C_pv(r)^{s+1}/(s+1)$ is decreasing and vanishing at infinity and thus nonnegative. Thus we have 
\be\label{ode2}
-v'(r)\ge C v(r)^{(s+1)/2},\ r\in(1,+\infty),
\ee
where $C=\big(2C_p/(s+1)\big)^{1/2}$. Arguing by contradiction, suppose that $v$ is strictly positive on $(1,+\infty)$. Dividing both sides of \eqref{ode2} and integrating, we have
$$-v(r)^{(1-s)/2}\ge Ar+B,$$
where $A=2C/(1-s)$ and $B$ is determined by the initial datum $v(1)$. This cannot occur, since the left hand side is negative, while the right hand side goes to $+\infty$, as $r\to+\infty$. 

We now prove the result for a generic compactly supported and nonnegative $f\in L^q(\R^d)$. Since the solution $u_p$ of \eqref{J-ard} is nonnegative and is a weak solution of \eqref{eulagrd}, we have that on each ball $B_R\subset\R^d$, $u_p\le u$, where $u\in H^1(B_R)$ is the solution of
$$-\Delta u=f\ \hbox{in}\ B_R,\qquad u=u_p\ \hbox{on}\ \partial B_R.$$
Since $f\in L^{d/2}(\R^d)$, by \cite[Theorem 9.11]{gt} and a standard bootstrap argument on the integrability of $u$, we have that $u$ is continuous on $B_{R/2}$. As a consequence, $u_p$ is locally bounded in $\R^d$. In particular, it is bounded since $u_p\wedge M$, where $M=\|u_p\|_{L^\infty(B_1)}$, is a better competitor than $u_p$ in \eqref{J-ard}. Let $w$ be a radially decreasing minimizer of \eqref{J-ard} with $f=\chi_{B_1}$. Thus $w$ is a solution of the PDE
$$-\Delta w+C_p w^s=\chi_{B_1},$$
in $\R^d$, where $C_p$ is as in \eqref{ode1}. Then, the function $w_t(x)=t^{2/(1-s)}w(x/t)$ is a solution of the equation
$$-\Delta w_t+C_p w^s_t=t^{2s/(1-s)}\chi_{B_t}.$$
Since $u_p$ is bounded, there exists some $t\ge1$ large enough such that $w_t\ge u_p$ on the ball $B_t$. Moreover, $w_t$ minimizes \eqref{J-ard} with $f=t^{2s/(1-s)}\chi_{B_t}$ and so $w_t\ge u_p$ on $\R^d$ (otherwise $w_t\wedge u_p$ would be a better competitor in \eqref{J-ard} than $w_p$). The conclusion follows since, by the first step of the proof, $w_t$ has compact support.
\end{proof}

The problems \eqref{lbrd} and \eqref{minrd} are similar both in the questions of existence and the qualitative properties of the solutions.

\begin{prop}\label{lbrdex}
For every $p>0$ there is a solution of the problem \eqref{lbrd} given by 
\be\label{V-prd2}
V_p=\left(\int_{\R^d}|u_p|^{2p/(p+1)}\,dx\right)^{1/p}|u_p|^{-2/(1+p)},
\ee
where $u_p$ is a radially decreasing minimizer of 
\begin{align}\label{J-ard2}
&\min\Bigg\{\int_{\R^d}|\nabla u|^2\,dx+\left(\int_{\R^d}|u|^{2p/(p+1)}\,dx\right)^{(p+1)/p}\ :\ u\in H^1(\R^d),\ \int_{\R^d}u^2\,dx=1\Bigg\}.
\end{align}
Moreover, $u_p$ has a compact support, hence the set $\{V_p<+\infty\}$ is a ball of finite radius in $\R^d$.
\end{prop}

\begin{proof}
Let us first show that the minimum in \eqref{J-ard2} is achieved. Let $u_n\in H^1(\R^d)$ be a minimizing sequence of positive functions normalized in $L^2$. Note that by the P\'olya-Szeg\"o inequality we may assume that each of these functions is radially decreasing in $\R^d$ and so we will use the identification $u_n=u_n(r)$. In order to prove that the minimum is achieved it is enough to show that the sequence $u_n$ converges in $L^2(\R^d)$. Indeed, since $u_n$ is a radially decreasing minimizing sequence, there exists $C>0$ such that for each $r>0$ we have
$$u_n(r)^{2p/(p+1)}\le\frac{1}{|B_r|}\int_{B_r}u_n^{2p/(p+1)}\,dx\le\frac{C}{r^d}.$$
Thus, for each $R>0$, we obtain
\be\label{lbrdex1}
\int_{B_R^c}u_n^2\,dx\le C_1\int_R^{+\infty}r^{-d(p+1)/p}\,r^{d-1}\,dr=C_2R^{-1/p},
\ee
where $C_1$ and $C_2$ do not depend on $n$ and $R$. Since the sequence $u_n$ is bounded in $H^1(\R^d)$, it converges locally in $L^2(\R^d)$ and, by \eqref{lbrdex1}, this convergence is also strong in $L^2(\R^d)$. Thus, we obtain the existence of a radially symmetric and decreasing solution $u_p$ of \eqref{J-ard2} and so, of an optimal potential $V_p$ given by \eqref{V-prd2}. 

We now prove that the support of $u_p$ is a ball of finite radius. By the radial symmetry of $u_p$ we can write it in the form $u_p(x)=u_p(|x|)=u_p(r)$, where $r=|x|$. With this notation, $u_p$ satisfies the equation:
$$-u_p''-\frac{d-1}{r}u_p'+C_pu_p^s=\lambda u_p,$$
where $s=(p-1)/(p+1)<1$ and $C_p>0$ is a constant depending on $p$. Arguing as in Proposition \ref{cs1}, we obtain that, for $r$ large enough,
$$-u_p'(r)\ge \left(\frac{C_p}{s+1}u_p(r)^{s+1}-\frac{\lambda}{2}u_p(r)^2\right)^{1/2}\ge \left(\frac{C_p}{2(s+1)}u_p(r)^{s+1}\right)^{1/2},$$
where, in the last inequality, we used the fact that $u_p(r)\to0$, as $r\to\infty$, and $s+1<2$. Integrating both sides of the above inequality, we conclude that $u_p$ has a compact support. In Figure \ref{fig2} we show the case $d=1$ and $f=\chi_{(-1,1)}$.
\end{proof}

\begin{figure}[h]
\begin{center}
\includegraphics[scale=1]{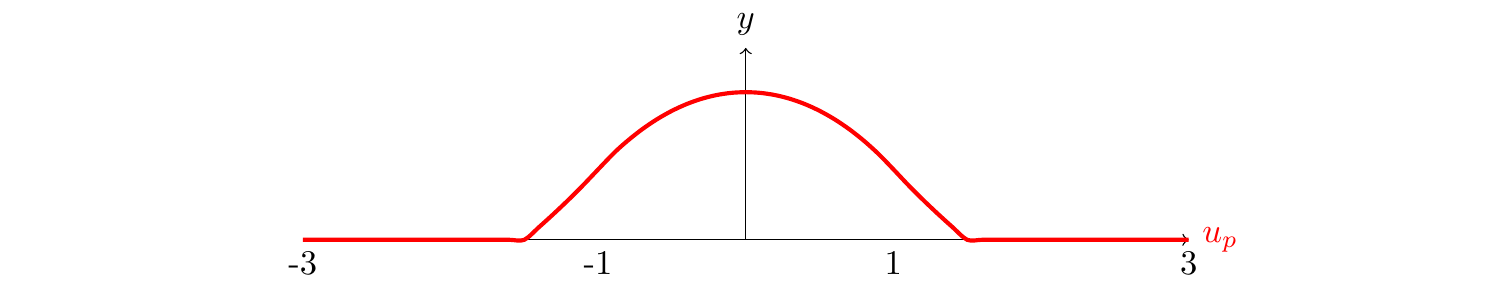}
\caption{The solution $u_p$ of problem \eqref{J-ard}, with $p>1$ and $f=\chi_{(-1,1)}$.}\label{fig2}
\end{center}
\end{figure}

\begin{oss}\label{igns} 
We note that the solution $u_p\in H^1(\R^d)$ of \eqref{J-ard2} is the function for which the best constant $C$ in the interpolated Gagliardo-Nirenberg-Sobolev inequality
\be\label{igns1}
\|u\|_{L^2(\R^d)}\le C\|\nabla u\|_{L^2(\R^d)}^{d/(d+2p)}\|u\|_{L^{2p/(p+1)}(\R^d)}^{2p/(d+2p)}
\ee
is achieved. Indeed, for any $u\in H^1(\R^d)$ and any $t>0$, we define $u_t(x):=t^{d/2}u(tx)$. Thus, we have that $\|u\|_{L^2(\R^d)}=\|u_t\|_{L^2(\R^d)}$, for any $t>0$. Moreover, up to a rescaling, we may assume that the function $g:(0,+\infty)\to\R$, defined by
\begin{align*}
g(t)&=\int_{\R^d}|\nabla u_t|^2\,dx+\left(\int_{\R^d}|u_t|^{2p/(p+1)}\,dx\right)^{(p+1)/p}\\
&=t^2\int_{\R^d}|\nabla u|^2\,dx+t^{-d/p}\left(\int_{\R^d}|u|^{2p/(p+1)}\,dx\right)^{(p+1)/p},
\end{align*}
achieves its minimum in the interval $(0,+\infty)$ and, moreover, we have
$$\min_{t\in(0,+\infty)}g(t)=C\left(\int_{\R^d}|\nabla u|^2\,dx\right)^{d/(d+2p)}\left(\int_{\R^d}|u|^{\frac{2p}{p+1}}\,dx\right)^{2(p+1)/(d+2p)},$$
where $C$ is a constant depending on $p$ and $d$. In the case $u=u_p$, the minimum of $g$ is achieved for $t=1$ and so, we have that $u_p$ is a solution also of
$$\min\Bigg\{\left(\int_{\R^d}|\nabla u|^2\,dx\right)^{d/(d+2p)}\left(\int_{\R^d}|u|^{2p/(p+1)}\,dx\right)^{2(p+1)/(d+2p)} :\ u\in H^1(\R^d),\ \int_{\R^d}u^2\,dx=1\Bigg\},$$
which is just another form of \eqref{igns1}.
\end{oss}

\section{Further remarks and open questions}\label{s6}

We recall (see \cite{bubu12}) that the injection $H^1_V(\R^d)\hookrightarrow L^2(\R^d)$ is compact whenever the potential $V$ satisfies $\int_{\R^d}V^{-p}\,dx<+\infty$ for some $0<p\le1$. In this case the spectrum of the Schr\"odinger operator $-\Delta+V$ is discrete and we denote by $\lambda_k(V)$ its eigenvalues. The existence of an optimal potential for spectral optimization problems of the form 
\be\label{lbrdk}
\min\left\{\lambda_k(V)\ :\ V\ge0,\ \int_{\R^d}V^{-p}\,dx\le1\right\},
\ee
for general $k\in\N$, cannot be deduced by the direct methods used in Subsection \ref{ss52}. In this last section we make the following conjectures:
\begin{enumerate}
[{\bf Conjecture 1)}]
\item For every $k\ge 1$, there is a solution $V_k$ of the problem \eqref{lbrdk}.
\item The set of finiteness $\{V_k<+\infty\}$, of the optimal potential $V_k$, is bounded.
\end{enumerate}

In what follows, we prove an existence result in the case $k=2$. We first recall that, by Proposition \ref{lbrdex}, there exists optimal potential $V_p$, for $\lambda_1$, such that the set of finiteness $\{V_p<+\infty\}$ is a ball. Thus, we have a situation analogous to the Faber-Krahn inequality, which states that the minimum  
\be\label{fk}
\min\left\{\lambda_1(\O)\ :\ \O\subset\R^d,\ |\O|=c\right\},
\ee
is achieved for the ball of measure $c$. We recall that, starting from \eqref{fk}, one may deduce, by a simple argument (see for instance \cite{hen03}), the Krahn-Szeg\"o inequality, which states that the minimum
\be\label{ks}
\min\left\{\lambda_2(\O)\ :\ \O\subset\R^d,\ |\O|=c\right\},
\ee
is achieved for a disjoint union of equal balls. In the case of potentials one can find two optimal potentials for $\lambda_1$ with disjoint sets of finiteness and then apply the argument from the proof of the Krahn-Szeg\"o inequality. In fact, we have the following result.
 
\begin{prop}\label{potks}
There exists an optimal potential, solution of \eqref{lbrdk} with $k=2$. Moreover, any optimal potential is of the form $\min\{V_1,V_2\}$, where $V_1$ and $V_2$ are optimal potentials for $\lambda_1$ which have disjoint sets of finiteness $\{V_1<+\infty\}\cap\{V_2<+\infty\}=\emptyset$ and are such that $\int_{\R^d} V_1^{-p}\,dx=\int_{\R^d} V_2^{-p}\,dx=1/2$. 
\end{prop}

\begin{proof}
Given $V_1$ and $V_2$ as above, we prove that for every $V:\R^d\to[0,+\infty]$ with $\int_{\R^d} V^{-p}\,dx=1$, we have
$$\lambda_2(\min\{V_1,V_2\})\le \lambda_2(V).$$
Indeed, let $u_2$ be the second eigenfunction of $-\Delta+V$. We first suppose that $u_2$ changes sign on $\R^d$ and consider the functions $V_+=\sup\{V,\infty_{\{u_2\le0\}}\}$ and $V_-=\sup\{V,\infty_{\{u_2\ge0\}}\}$ where, for any measurable $A\subset\R^d$, we set
$$\infty_A(x)=
\begin{cases}
\ds +\infty,\qquad x\in A,\\
0, \qquad x\notin A. 
\end{cases}$$
We note that
$$1=\int_{\R^d}V^{-p}\,dx=\int_{\R^d}V^{-p}_+\,dx+\int_{\R^d}V^{-p}_-\,dx.$$
Moreover, on the sets $\{u_2>0\}$ and $\{u_2<0\}$, the following equations are satisfied:
$$-\Delta u_2^++V_+u_2^+=\lambda_2(V)u_2^+,\qquad -\Delta u_2^-+V_-u_2^-=\lambda_2(V)u_2^-,$$
and so, multiplying respectively by $u_2^+$ and $u_2^-$, we obtain that
\be\label{t61}
\lambda_2(V)\ge\lambda_1(V_+),\qquad \lambda_2(V)\ge \lambda_1(V_-),
\ee
where we have equalities if, and only if, $u_2^+$ and $u_2^-$ are the first eigenfunctions corresponding to $\lambda_1(V_+)$ and $\lambda_1(V_-)$. Let now $\widetilde V_+$ and $\widetilde V_-$ be optimal potentials for $\lambda_1$ corresponding to the constraints
$$\int_{\R^d} \widetilde V_+^{-p}\,dx=\int_{\R^d} V_+^{-p}\,dx,\qquad\int_{\R^d} \widetilde V_-^{-p}\,dx=\int_{\R^d} V_-^{-p}\,dx.$$
By Proposition \ref{lbrdex}, the sets of finiteness of $\widetilde V_+$ and $\widetilde V_-$ are compact, hence we may assume (up to translations) that they are also disjoint. By the monotonicity of $\lambda_1$, we have
$$\max\{\lambda_1(V_1),\lambda_1(V_2)\}\le\max\{\lambda_1(\widetilde V_+),\lambda_1(\widetilde V_-)\},$$
and so we obtain
 $$\lambda_2(\min\{V_1,V_2\})\le\max\{\lambda_1(\widetilde V_+),\lambda_1(\widetilde V_-)\}\le\max\{\lambda_1(V_+),\lambda_1(V_-)\}\le\lambda_2(V),$$
as required. If $u_2$ does not change sign, then we consider $V_+=\sup\{V,\infty_{\{u_2=0\}}\}$ and $V_-=\sup\{V,\infty_{\{u_1=0\}}\}$, where $u_1$ is the first eigenfunction of $-\Delta+V$. Then the claim follows by the same argument as above. 
\end{proof}

For more general cost functionals $F(V)$, the question if the optimization problem
$$\min\left\{F(V)\ :\ V\ge0,\ \int_{\R^d}V^p\,dx\le1\right\}$$
admits a solution is, as far as we know, open.


%
%
%

\end{document}